\documentclass{article}
\usepackage[utf8]{inputenc}

\usepackage{lmodern}

\usepackage{mathscinet}     
\usepackage{a4wide,amsmath}        
\usepackage{amsfonts}       
\usepackage{amsthm}         
\usepackage{bbding}         
\usepackage{bm}             
\usepackage{graphicx}       
\usepackage{fancyvrb}       
\usepackage[sort&compress,numbers]{natbib}         
\usepackage[nottoc]{tocbibind} 

\usepackage{dcolumn}        
\usepackage{booktabs}       
\usepackage{paralist}       
\usepackage{todonotes}
\usepackage{titling}
\usepackage{hyperref}
\usepackage[framemethod=tikz]{mdframed}

\newtheorem{theorem}{Theorem}[section]

\theoremstyle{definition}
\newtheorem{definition}[theorem]{Definition}

\theoremstyle{plain}
\newtheorem{proposition}[theorem]{Proposition}

\theoremstyle{plain}
\newtheorem{lemma}[theorem]{Lemma}

\theoremstyle{plain}
\newtheorem{corollary}[theorem]{Corollary}

\theoremstyle{plain}
\newtheorem{remark}[theorem]{Remark}

\def\diff{\mathrm{d}}

\def\nat{\mathbb{N}}
\def\hra{\hookrightarrow}

\def\measurable{\mathcal{M}}
\def\measurablep{\mathcal{M}_+}
\def\quasilorentz{m_\varphi}

\newcommand{\EofXint}[1]{\frac{t}{\rVert E_t \rVert_#1}}

\newcommand{\smallmnormdefof}[2]{
\sup_{t \in (0, \infty)} #1 #2}
\def\varphinorm{\lVert\cdot\rVert_{M_\varphi}}
\def\varphinormq{\lVert\cdot\rVert_{\quasilorentz}}

\def\varphinormqdef{\sup_{t\in(0,\infty)}\varphi(t)f^{*}(t)}
\newcommand{\varphinormof}[1]{\lVert#1\rVert_{M_\varphi}}
\newcommand{\varphinormqof}[1]{\lVert#1\rVert_{\quasilorentz}}
\newcommand{\varphinormdefof}[1]{\sup_{t\in(0,\infty)}\varphi(t)#1^{**}(t)}
\newcommand{\varphinormqdefof}[1]{\sup_{t\in(0,\infty)}\varphi(t)#1^{*}(t)}

\def\aoperator{\int_0^\infty a(st)f(s) \diff s}

\begin{document}

\title{Optimality of function spaces for kernel integral operators}

\author{Jakub Tak\'a\v c
     }

\date{November 2022}

\begin{titlingpage}
    \maketitle
    \begin{abstract}
         We explore boundedness properties of kernel integral operators acting on rearrangement-invariant (r.i.)~spaces. In particular, for a given r.i.~space $X$ we characterize its optimal range partner, that is, the smallest r.i.~space $Y$ such that the operator is bounded from $X$ to $Y$. We apply the general results to Lorentz spaces to illustrate their strength.
    \end{abstract}
\end{titlingpage}

\bibliographystyle{abbrv.bst}

\section*{Introduction}
\addcontentsline{toc}{chapter}{Introduction}
The \textit{Laplace transform} is defined by
\[
\mathcal L f(t)=\int_0\sp{\infty}f(s)e\sp{-ts}\,\diff s
\]
for every real-valued function $f$ on $(0,\infty)$ for which the integral makes sense, and for every $t\in(0,\infty)$. It is well known that $\mathcal L$ is an important integral operator with plenty of applications for example in the theory of differential equations, probability theory, investigation of spectral properties of pseudo-differential operators or the study of Fredholm integral equations (see for instance \cite{Andersen2, N, Sharpley, BoSh, MW}). The Laplace transform can be viewed as a~particular instance of a~fairly more general class of \textit{kernel} integral operators
\[
Kf(t)=\int_0\sp{\infty}f(s)k(s,t)\,\diff s,
\]
where $k$ is an appropriate measurable function of two variables (to obtain the Laplace transform we set $k(s,t)=e^{-st}$ for $s,t\in(0,\infty)$).

In this text we focus on problems concerning sharp action of kernel operators of a particular type, namely operators defined by
\[
S_af(t)=\int_0\sp{\infty}f(s)a(st)\,\diff s,
\]
for all real-valued functions $f$ on $(0,\infty)$ for which the integral makes sense, where $a\colon(0,\infty)\to[0,\infty)$ is an appropriate function of one variable. Various particular types of such operators and their numerous modifications have been studied by many authors. To name just a few classical ones, see \cite{Andersen,Carro,Lai} and the references therein.

Our goal here is to investigate fine properties of operators of type $S_a$ acting on the so-called rearrangement-invariant (r.i.~for short) spaces. In these spaces, sometimes in literature called also symmetric spaces, or K\"othe spaces, the decisive parameter is always the~\textit{size} of a~function (rather than other properties such as continuity, smoothness or oscillation). They are built upon an idea of axiomatization of good properties enjoyed by Lebesgue spaces, which constitute a central subclass of r.i.~spaces, but in themselves are not rich enough in order to satisfactorily describe action of important operators and embeddings, especially in their various critical or limiting states. With a~certain licence it can be said that the norms in r.i.~spaces take into account only the measure of level sets of a~given function.

To be more precise, we will concentrate on the questions of existence and eventual characterization of the optimal partner space for an operator of type $S_a$. By the optimal range partner space, supposing a domain space $X$ is given, we mean a space $Y$ such that $S_a$ takes boundedly $X$ to $Y$, and $Y$ is the smallest possible such space. Analogously we define the optimal domain partner when the range space is fixed. To make such questions sensible one has to, of course, state the pool of competing spaces - here this will always be that of r.i. spaces.

Let us note that while action of kernel operators on function spaces has been studied, little attention has been so far paid to optimality of such results. On the other hand an extensive research of optimality of function spaces in different situations (e.g. in Sobolev embeddings, trace embeddings, Gaussian measure space, probability spaces etc., see for example \cite{Cianchi, Bravo, CM1, CP-TAMS, CR1, CwPu, DS, EKP, KP2, Mus1, Mus2, P}) could be seen mainly during last two decades. There are few exceptions, namely papers in which optimality of spaces for integral operators is studied, for example \cite{BEP}, \cite{EMMP} or \cite{STZ}, \cite{P}.

To fulfil our goal, we will combine known techniques with certain new ones which we have to develop. In particular, we will calculate the Peetre $K$ functional for certain specific pairs of spaces (this has been known only partially) and we will, in an efficient way, combine the notion of the Marcinkiewicz-type endpoint space with the norm of the dilation operator. Such methods have not been known before and we shall show that they lead to quite a fruitful theory.

The paper is structured as follows. In Section~\ref{CH:1} we fix notation and collect all the preliminary stuff including all the definitions and basic knowledge about the function spaces, operators and related topics. We give a detailed definition of r.i. spaces and recall everything we shall need to know about them. We also define particular function spaces which will be used in examples to illustrate the abstract results. In Section~\ref{CH:2} we present several background results that will be needed in the proofs of the main results. In particular, we introduce here two types of Marcinkiewicz endpoint spaces and study their elementary but useful properties. We also establish certain important relations concerning the Peetre $K$-functional in the spirit of~\cite{Milman}.

Main results are stated and proved in~Section~\ref{CH:3}. We tackle the problem of the very existence of the partner space, and in the affirmative cases we characterize or construct the optimal one. Furthermore we present a series of results based on spaces given in terms of the norm of the dilation operator on the domain space. To this end we establish several results of independent interest, providing Calder\'on-type estimates for the operators in question. The final section contains a~thorough and comprehensive analysis of the action of operators $S_a$ on Lorentz spaces.

\section{Preliminaries}\label{CH:1}

In this section we recall some definitions and basic properties of
rearrangement-invariant spaces. The standard reference is \cite{BS}.

We denote by $m$ the Lebesgue measure on $(0,\infty)$ and define
\[
	\measurable = \{\hbox{$f:\, (0,\infty) \to [-\infty , \infty]$: $f$ is 			Lebesgue-measurable in $(0,\infty)$}\},
\]
and
\[
	\measurablep = \{f \in \measurable:\, f \geq 0\}.
\]

The \textit{distribution function} $f_*: (0,\infty) \rightarrow [0, \infty]$ of a function $f \in \measurable$ is defined as
\[
	f_*(\lambda) = |\{x\in (0,\infty):\, |f(x)| > \lambda \}|, \, \lambda \in (0,\infty),
\]
where $|\cdot|$ denotes the Lebesgue measure,
and the \textit{non-increasing rearrangement} $f\sp* : (0,\infty) \to [0, \infty ]$ of a function $f\in \measurable$  is defined as
\[
f^*(t)=\inf\{\lambda\in(0,\infty):\, f_*(\lambda) \leq t\},\, t\in(0,\infty).
\]
The operation $f\mapsto f^*$ is monotone in the sense that $|f|\leq |g|$ a.e. in $(0,\infty)$ implies $f^*\leq g^*$.
We define the \textit{elementary maximal function}
${f^{**}: (0,\infty) \rightarrow [0, \infty]}$ of a function
$f \in \measurable$ as
\[
	f^{**}(t) = \frac{1}{t} \int_0^t f^*(s) \diff s.
\]
While the operation $f \mapsto f^{**}$ is subadditive, that is, for any $f,g \in \measurable$ and $t\in(0,\infty)$ one has
\begin{equation}\label{E:TwoStar-subadditivity}
	(f+g)^{**}(t) \leq f^{**}(t) + g^{**}(t),
\end{equation}
for $f \mapsto f^*$ one only has the following property.
Let $s,t\in(0,\infty)$ and $f,g \in \measurable$, then
\begin{equation}\label{E:dec-rearrangement-weak-subaditivity}
(f+g)^*(s+t) \leq f^*(t) + g^*(s).
\end{equation}
We recall that for every $f\in\measurablep$ and every $t\in(0,\infty)$, one has
\begin{equation}\label{E:estimate-of-distribution-function}
	|\{s\in(0,\infty):\,f(s)>f^*(t)\}|\leq t.
\end{equation}
The \textit{Hardy-Littlewood inequality} asserts that if $f, g \in \measurable$,
then
\begin{equation}\label{E:HL}
\int _0\sp{\infty} |f(t) g(t)| \,\diff t \leq \int _0\sp{\infty} f\sp*(t) g\sp*(t) \diff t.
\end{equation}
Let us fix $f_1,f_2\in\measurablep$. If for any $t\in (0,\infty)$ it holds that
\begin{equation*}
    \int_0^t f_1(s) \diff s \leq \int_0^t f_2(s) \diff s,
\end{equation*}
then Hardy's Lemma asserts that for any $g\in\measurablep$ non-increasing, the following inequality holds,
\begin{equation}
    \label{Hardy-Lemma}
    \int_0^\infty f_1(s)g(s)\diff s \leq \int_0^\infty f_2(s)g(s) \diff s.
\end{equation}
Hardy's inequality asserts that if $p\in(1,\infty)$ and $\alpha>-(1-\tfrac{1}{p})$, then
\[
\int_0^\infty \left(t^{-\alpha-1} \int_0^t f(s) \diff s\right)^p \diff t
\leq \left(\frac{1}{1+\alpha-\tfrac{1}{p}}\right)^p \int_0^\infty \left(f(t)t^{-\alpha}\right)^p \diff t
\quad \text{for $f\in\measurablep$,}
\]
and if $\alpha>1-\tfrac{1}{p}$, then
\[
\int_0^\infty \left(t^{-\alpha+1} \int_t^\infty f(s) \diff s\right)^p \diff t
\leq \left(\frac{1}{-1+\alpha+\tfrac{1}{p}}\right)^p \int_0^\infty \left(f(t)t^{\alpha}\right)^p \diff t
\quad \text{for $f\in\measurablep$.}
\]
This inequality can be found for example in \cite[Chapter 3, Lemma 3.9]{BS}, where a slightly different, but equivalent, formulation is used.

Following \cite{BS} we say that a functional $\rho: \measurablep \to [0,\infty]$ is a \textit{Banach function norm}, if for all $f$, $g$ and
$\{f_j\}_{j\in\nat}$ in $\measurablep$, and every $\lambda \geq 0$, the
following properties hold:
\begin{itemize}
\item[(P1)] $\rho(f)=0$ if and only if $f=0$;
$\rho(\lambda f)= \lambda\rho(f)$; \par\noindent
\qquad $\rho(f+g)\leq \rho(f)+ \rho(g)$;
\item[(P2)] $  f \le g$ a.e.\  implies $\rho(f)\le\rho(g)$;
\item[(P3)] $  f_j \nearrow f$ a.e.\ implies
$\rho(f_j) \nearrow \rho(f)$;
\item[(P4)] $\rho(\chi_G)<\infty$ \ for every $G\subset(0,\infty)$ of finite measure;
\item[(P5)] for every $G\subset(0,\infty)$ of finite measure there is a constant $C_G$ such that $\int_{G} f(t)\diff t \leq C_G \rho(f)$.
\end{itemize}
If also the property
\begin{itemize}
	\item[(P6)] $\rho(f) = \rho(g)$ whenever
	$f^* = g^*$,
\end{itemize}
holds, we say that $\rho$ is a
\textit{rearrangement-invariant Banach function norm}, or just a
\textit{rearrangement-invariant norm}.
If $\rho$ is a rearrangement-invariant norm, then the collection
\[
	X=X(\rho)=\{f\in\measurable:\, \rho(|f|)<\infty\}
\]
is called a \textit{rearrangement-invariant Banach function space}, or just a \textit{rearrangement-invariant space}. The norm on the space $X$ is given by
$\lVert f \rVert_X = \rho(|f|)$. Note that $\rho(|f|)$ is defined for every $f\in\measurable$, and
\[
	f\in X\quad\Leftrightarrow\quad \rho(|f|)<\infty.
\]
We recall that it follows from the axioms that such $X$ is always complete (even if (P6) does not hold), hence the name Banach function space.
For a rearrangement-invariant norm $\rho$ we define its \textit{associate norm} by
\[
	\rho'(g) = \sup\left\{\int_0^\infty f(t)g(t)\diff t:\,f\in\measurablep,\, \rho(f)\leq 1\right\} \quad \text{for $g \in \measurablep$}
\]
This $\rho'$ is also a rearrangement-invariant norm.
Furthermore, it also holds that $\rho'' = \rho$.
If $X = X(\rho)$ is a rearrangement-invariant space and $\rho'$ is the norm associate to $\rho$, then $X(\rho')$ is the \textit{associate space}
of $X$ and is denoted by $X'$.

If $X, Y$ are rearrangement-invariant spaces, we denote by $X\hra Y$ the continuous embedding of $X$ into $Y$ and by $T\colon X\to Y$ the boundedness of an operator $T$ from $X$ to $Y$. We have
\begin{equation}\label{I:Asociate-embedding}
X\hra Y \iff Y' \hra X'.
\end{equation}
Given a linear operator $T$ defined on some subspace of $\measurable$ we call the operator $T'$ adjoint operator to the operator $T$ if 
\begin{equation*}
    \int_0^\infty T(f)g = \int_0^\infty fT'(g)
\end{equation*}
for all $f,g\in\measurable$ for which the left hand side makes sense. Recall that $T\colon X \to Y$ holds if and only if $T'\colon Y' \to X'$.
We say that the rearrangement-invariant space $Y$ is \textit{the optimal range partner} for the linear operator $T$ and a given domain rearrangement-invariant space $X$ if $T\colon X \rightarrow Y$ and for every rearrangement-invariant space $Z$ such that $T\colon  X \rightarrow Z$ it holds that $Y \hra Z$. An operator which will be used extensively throughout this work is the dilation operator $E_t$ defined for any $t \in (0,\infty)$ by the formula
\[
    E_t f (s) = f\left(\frac{s}{t}\right).
\]
We recall that $E_t$ is bounded on every rearrangement-invariant space for any $t \in (0,\infty)$.

We define the \textit{fundamental function}, $\varphi_X$, of a~given rearrangement-invariant space $X$ by $\varphi_X(t)=\|\chi_{(0,t)}\|_X$, $t\in(0,\infty)$.
We say that a function $\varphi\colon (0,\infty) \rightarrow (0,\infty)$ is \textit{quasiconcave} if it is non-decreasing and
$\frac{t}{\varphi(t)}$ is non-decreasing.
We say that the function $\varphi\colon (0,\infty) \rightarrow (0,\infty)$ satisfies the $\Delta_2$ condition, if it is non-decreasing and there exists a constant $C>0$ such that $\varphi(2t)\leq C\varphi(t)$ for all $t>0$.
The fundamental function of any rearrangement-invariant space is quasiconcave. Given a quasiconcave function~$\varphi$, we define the~rearrangement-invariant spaces $M_\varphi, \Lambda_\varphi$ with the rearrangement-invariant norms given by
\begin{equation}\label{D:Marcinkiewitz}
\varphinormof{f} = \varphinormdefof{f}, \, f \in \measurable,
\end{equation}
and
\begin{equation}\label{D:Lorentz}
\lVert f \rVert_{\Lambda_\varphi} =
\int_0^\infty f^*(t)\diff \varphi(t), \, f \in \measurable.
\end{equation}
These are indeed rearrangement-invariant norms and the spaces are called the  \textit{Lorentz endpoint space} and the \textit{Marcinkiewicz endpoint space} respectively. It is also known that both $M_\varphi$ and
$\Lambda_\varphi$ have a~common fundamental function which is equal to $\varphi$.
For a rearrangement-invariant space $X$ we denote
\[
	M(X) = M_{\varphi_X}, \,
	\Lambda(X) = \Lambda_{\varphi_X},
\]
where $\varphi_X$ is the fundamental function of $X$. We recall that for a rearrangement-invariant space $X$ we have
\[
	\Lambda(X) \hra X \hra M(X),
\]
with norms of both embeddings equal to 1.
In other words, the spaces $M_\varphi, \Lambda_\varphi$ are respectively the largest and the smallest rearrangement-invariant space with the fixed fundamental function equal to $\varphi$.

One of the basic examples of an important class of rearrangement-invariant spaces can be obtained by considering general Lorentz $L^{p,q}$ spaces with $p,q \in (0,\infty]$, governed for $f \in\measurable$ by the functional
\[
\rho_{p,q}(f)=
\begin{cases}
\left(\int_0^{\infty}(t^{\tfrac{1}{p}}f^*(t))^q\,\frac{\diff t}{t}\right)^{\frac{1}{q}} &\textup{if}\ 0<q<\infty,\\
\sup_{t\in(0,\infty)}t^{\tfrac{1}{p}}f^*(t)&\textup{if}\ q=\infty.
\end{cases}
\]
We recall that these are equivalent to rearrangement-invariant norms in cases when either $p\in(1,\infty)$ and $q\in[1,\infty]$ or $p=q=1$ or $p=q=\infty$. In case $p=\infty$ and $q<\infty$ the resulting space is a trivial set containing only the zero function. Furthermore, note that $L^{p,p} = L^p$, where $L^p$ is the classical Lebesgue space. For $p\in[1,\infty]$ we define the associated exponent $p'$ by $\tfrac{1}{p}+\tfrac{1}{p'} = 1$. The following equalities hold up to an equivalence of norms and further in this paper will be treated as equalities, disregarding the equivalence constants.
\begin{equation*}
\begin{split}
    &L^{p,1} = \Lambda(L^p)\quad \text{for $p\in[1,\infty)$},\\
    &L^{p,\infty} = M(L^p)\quad \text{for $p\in(1,\infty]$},\\
    &(L^{p,q})' = L^{p',q'} \quad \text{for $p\in(1,\infty)$, $q\in[1,\infty]$ or $p=q=\infty$ or $p=q=1$}.
\end{split}
\end{equation*}
We will also be using the following embeddings:

(a) if $q_1,q_2\in[1,\infty]$ and $q_1\leq q_2$, then
\begin{equation*}
    L^{p,q_1} \hookrightarrow L^{p,q_2}\quad\text{for all $p\in(1,\infty)$,}
\end{equation*}

(b) if $q_1,q_2\in[1,\infty]$, $p_1,p_2\in (1,\infty)$ and $p_1< p_2$, then
\begin{equation*}
    L^1\cap L^{p_1,q_1} \hookleftarrow L^1\cap L^{p_2,q_2} \quad\text{and \quad $L^\infty \cap L^{p_1,q_1} \hookrightarrow L^\infty \cap L^{p_2,q_2}$,}
\end{equation*}

(c) if $q_1,q_2\in[1,\infty]$, $p_1,p_2\in (1,\infty)$ and $p_1< p_2$, then
\begin{equation*}
    L^1 + L^{p_1,q_1} \hookrightarrow L^1 + L^{p_2,q_2} \quad\text{and \quad $L^\infty + L^{p_1,q_1} \hookleftarrow L^\infty + L^{p_2,q_2}$,}
\end{equation*}
where, as per usual, for $X_0$ and $X_1$ rearrangement-invariant spaces, the norm in $X_0 \cap X_1$ is given by $\max\{\lVert \cdot \rVert_{X_0},\lVert \cdot \rVert_{X_1}\}$. For $X_0+X_1$ see below.

Let $X_0$ and $X_1$ be quasi-normed spaces, which are \textit{compatible} in the sense that
they are embedded in some common Hausdorff topological vector space (in our case we are working with the space $\{f \in \measurable,\, |f|<\infty\; \text{a.e.}\}$).
By $X_0+X_1$ we denote the set of all functions $f\in\measurable$ for which there exists a~decomposition $f=g+h$ such that $g\in X_0$ and $h\in X_1$. We equip the space $X_0+X_1$ with the quasinorm
\[
\|f\|_{X_0+X_1}=\inf_{f=g+h}(\|g\|_{X_0}+\|h\|_{X_1}),
\]
where the infimum is taken over all such decompositions. Recall that if $X_0$ and $X_1$ are normed, then $X_0+X_1$ is also normed. For $f\in X_0+X_1$ the \textit{Peetre $K$-functional} is defined by
$$
K(t,f;X_0,X_1):=\inf_{f=g+h}\left(\|g\|_{X_0}+t\|h\|_{X_1}\right)\quad\textup{for}\ t>0.
$$
The function $K$ as a function of variable $t$ is increasing and concave on $(0,\infty)$. Furthermore, the function $t\sp{-1}K(t,f;X_0,X_1)$ is non-increasing on $(0,\infty)$.
Observe that
\begin{equation}\label{E:Switching-Spaces-KFunc}
\frac{1}{t}K(f,t;X_0,X_1) = K(f,\frac{1}{t},X_1,X_0).
\end{equation}
Recall that in the case when $X_0=L\sp1$ and $X_1=L\sp{\infty}$, an exact formula for the $K$ functional is known (see e.g. \cite[Chapter 5, Theorem 1.6]{BS}), namely,
\begin{equation}\label{E:K-traditional}
K(f,t;L\sp1,L\sp{\infty})=\int_0\sp tf\sp*(s)\,ds\quad\textup{for}\ t\in(0,\infty)\ \textup{and}\ f\in(L\sp1+L\sp{\infty}).
\end{equation}

\section{Background results}\label{CH:2}

In this section, we shall establish some results on rearrangement-invariant spaces and $K$ functionals which will be useful later. Let us begin with the definition of a Marcinkiewicz-type space $\quasilorentz$ similar to $M_\varphi$.

\begin{definition}
\label{D:small-m-norm}
Let
$\varphi : (0,\infty) \rightarrow (0,\infty)$ be a function satisfying the $\Delta_2$ condition. We define the functional
$\varphinormq$ for $f \in \measurable$ by the formula

\begin{equation*}
\varphinormqof{f} = \varphinormqdef.
\end{equation*}
We define the space $\quasilorentz$ as the set of all functions $f$ for which the functional $\varphinormqof{f}$ is finite.
\end{definition}

The $\Delta_2$ condition imposed on the function $\phi$ guarantees that $\varphinormqof{\cdot}$ is a quasinorm (see also \cite{CMMP}), therefore we will consider $\quasilorentz$ to be a quasinormed space.
We say that a linear set $M \subset \measurable$ together with a functional $F \colon M \rightarrow [0,\infty)$ can be equivalently renormed with a rearrangement-invariant norm, if there exists a rearrangement-invariant space $X$ and constants $C_1$ and $C_2$ such that $X=M$ and
\begin{equation*}
C_1 \lVert f \rVert_X \leq F(f) \leq C_2 \lVert f \rVert_X \quad \text{for $f \in X$}.
\end{equation*}
If that is the case, we will identify $M=X$ not only set-wise but also as spaces, in other words, we consider $M$ to be a rearrangement-invariant space.
Let $\varphi$ be a quasi-concave function. Then $\quasilorentz$ can be equivalently renormed with a rearrangement-invariant norm $\varphinorm$ if and only if there exists a constant $C>0$ such that
\begin{equation}
\label{RINormCondition}
\int_0^t \frac{1}{\varphi(s)} \diff s \leq C \frac{t}{\varphi(t)} \quad \text{for every $t\in(0,\infty)$.}
\end{equation}
One of the implications can be found for example in \cite[Proposition 7.10.5]{PKJF} and the other is an easy exercise. A more general result can be found in~\cite{CP-TAMS}.

Our next aim is to characterize the $K$-functional for a~pair of Marcinkiewicz spaces. Results in this direction can be found in literature (see e.g.~\cite{Torchinsky, Bennett, Ericsson, Milman}). Here we need to obtain $K(m_\varphi,m_\psi)$ with rather mild conditions on the functions $\varphi, \psi$. In particular, we require little more than that they satisfy the $\Delta_2$ condition.
Our proof of the Proposition \ref{P-Main-Kfunc} is in the spirit of that of~\cite[Theorem~4.1]{Milman}.

\begin{proposition}
\label{P-Main-Kfunc}
Let $\varphi$, $\psi$ $\in \measurablep$ satisfy the $\Delta_2$ condition. Denote by $s$ the function $s(t) = \frac{\varphi(t)}{\psi(t)}$ for $t \in (0,\infty)$. If $s$ is monotone, then there exist constants $C_1$, $C_2$ $>0$ such that for every $t>0$ and $f \in \measurable$
\begin{equation}\label{E:11}
    C_1(\lVert (E_\frac{1}{2} f)^*\chi_A \rVert_{m_\varphi} + t \lVert (E_\frac{1}{2} f)^*\chi_B \rVert_{m_\psi}) \leq
    K(f,t,m_\varphi,m_\psi) \leq
    C_2(\lVert f^*\chi_A \rVert_{m_\varphi} + t \lVert f^*\chi_B \rVert_{m_\psi}),
\end{equation}
where $A=\{u\colon s(u) < t \}$ and $B = (0,\infty) \setminus A$.
\end{proposition}

\begin{proof}
First, let $f = f_0 + f_1$, $f_0 \in m_\varphi$, $f_1 \in m_\psi$ and $t>0$. Then we have
\begin{equation}
\label{I1-I2-Inequality}
    \lVert f_1 \rVert_{m_\varphi} + t\lVert f_1 \rVert_{m_\psi} \geq \frac{1}{2} (I_1 + I_2),
\end{equation}
where
\begin{equation}
\label{I1-Definition}
    I_1 = \sup_{u>0} f_0^*(u)\chi_A(u)\varphi(u) + t\sup_{u>0} f_1^*(u)\chi_A(u)\psi(u)
\end{equation}
and
\begin{equation}
\label{I2-Definition}
    I_2 = \sup_{u>0} f_0^*(u)\chi_B(u)\varphi(u) + t\sup_{u>0} f_1^*(u)\chi_B(u)\psi(u).
\end{equation}
From the definition of the set $A$, we have
\begin{equation}
\label{I1-Second-Term-Estimate}
    t\sup_{u>0} f_1^*(u)\chi_A(u)\psi(u) \geq \sup_{u>0} f_1^*(u)\chi_A(u)\varphi(u),
\end{equation}
therefore, combining \eqref{I1-Definition} with \eqref{I1-Second-Term-Estimate}, we obtain
\begin{equation}
\label{I1-Estimate}
    \begin{split}
        I_1 \geq \sup_{u>0} (f_0^*+f_1^*)(u)\chi_A(u)\varphi(u)
            \geq \sup_{u>0} (f_0+f_1)^*(2u)\chi_A(u)\varphi(u)
            =\lVert (E_{\tfrac{1}{2}} f^*) \chi_A \rVert_{m_\varphi},
    \end{split}
\end{equation}
for some positive constant $C$.

Now from the definition of the set $B$, it is easy to see that
\begin{equation*}
    \sup_{u>0} f_0^*(u)\chi_B(u)\varphi(u) \geq t\sup_{u>0} f_0^*(u)\chi_B(u)\psi(u),
\end{equation*}
which in combination with \eqref{I2-Definition} immediately gives
\begin{equation}
\label{I2-Estimate}
    \begin{split}
        I_2 &\geq t\sup_{u>0} f_0^*(u)\chi_B(u)\psi(u) + t\sup_{u>0} f_1^*(u)\chi_B(u)\psi(u) \\
        &\geq t \sup_{u>0} (f_0^*+f_1^*)(u) \chi_B(u)\psi(u) \\
        &\geq t \sup_{u>0} (f_0^*+f_1^*)(2u) \chi_B(u)\psi(u) \\
        &=t\lVert (E_{\tfrac{1}{2}} f^*) \chi_B(u) \rVert_{m_\psi},
    \end{split}
\end{equation}
for some positive constant $C$. Since the decomposition $f = f_0 + f_1$ was arbitrary, combining \eqref{I1-Estimate} and \eqref{I2-Estimate} together with \eqref{I1-I2-Inequality} gives the first inequality.

For the second inequality we will assume $s$ to be non-decreasing and define
$s^\dagger(t) = \inf s^{-1}(t)$. If $s$ is non-increasing the proof is analogous with $s^\dagger$ defined using supremum. Decompose $f = f_0+f_1$, where
\begin{equation*}
    f_0= \begin{cases} f-f^*\left(s^{\dagger}(t)\right) & \text{if } f > f^*\left(s^{\dagger}(t)\right) \\
                      0                                    & \text{otherwise}.     %
     \end{cases}
\end{equation*}
Then we have
\begin{equation}
\label{E-Kfunc-First-Member}
\begin{split}
    t\lVert f_1^* \rVert_{m_\psi}
    &\leq t f^*(s^\dagger(t))\psi(s^\dagger(t)) + t\lVert f^*\chi_B \rVert_{m^\psi} \\
    &\leq \sup_{u>0} f^*(u)\chi_A(u)\varphi(u) + t\lVert f^*\chi_B \rVert_{m^\psi},
\end{split}
\end{equation}
where the last inequality is a consequence of continuity of $\psi$.
We also have
\begin{equation}
\label{E-Kfunc-Second-Member}
    \lVert f_0^* \rVert_{m_\varphi} \leq \sup_{u>0} f^*(u)\chi_A(u)\varphi(u)
\end{equation}
Combining \eqref{E-Kfunc-First-Member} and \eqref{E-Kfunc-Second-Member} gives the second inequality.
\end{proof}

Let us point out that for the first inequality in~\eqref{E:11} the assumption of monotonicity of $s$ is not needed.


In the proofs of the main results we will also require a characterization of the $K$-functional for a Marcinkiewicz space and $L^{\infty}$. Again, various versions of this result are known and scattered in literature but not in the precise form which we need. Therefore, for a reader's convenience we insert the proof.

\begin{proposition}
\label{P-Linf-Smallm-K-functional}
Let $\varphi\colon(0,\infty)\to(0,\infty)$ be a~strictly increasing, left continuous function satisfying the $\Delta_2$ condition. Then
\begin{equation}\label{E:K-inequality}
    \varphinormqof{\chi_{(0,\varphi^{-1}(t))}f^*} \leq
    K(f,t;\quasilorentz,L^\infty) \leq
    2\varphinormqof{\chi_{(0,\varphi^{-1}(t))}f^*}
\end{equation}
for every $f\in\measurable$ and every $t\in(0,\infty)$.
\end{proposition}
\begin{proof}
Let $f \in (\quasilorentz + L^\infty)$ and $t>0$. Both $L^\infty$ and $\quasilorentz$ norms are defined in terms of $f^*$ so it will suffice to prove the assertion assuming that $f \geq 0$. First, decompose $f=f_0+f_1$, where

\begin{equation*}
f_0= \begin{cases} f-f^*\left(\varphi^{-1}(t)\right) & \text{if } f > f^*\left(\varphi^{-1}(t)\right) \\
                      0                                    & \text{otherwise}.     %
     \end{cases}
\end{equation*}
Then since $\varphi$ is left continuous, we have
\begin{equation*}
\begin{split}
\sup_{0<s<\varphi^{-1}(t)}f^*(s)\varphi(s)
    & \geq \lim_{s\to{\varphi^{-1}(t)}_-} f^*(s)\varphi(s)
        \\
    & = \lim_{s\to{\varphi^{-1}(t)}_-} f^*(s) \lim_{s\to{\varphi^{-1}(t)}_-} \varphi(s)
        \\
    & \geq f^*(\varphi^{-1}(t))\varphi(\varphi^{-1}(t)) = f^*(\varphi^{-1}(t)) t.
\end{split}
\end{equation*}
And so from the definition of $f_0$ and the above calculation
\begin{equation}\label{e1}
\begin{split}
t {\lVert f_1 \rVert}_\infty \leq f^*(\varphi^{-1}(t)) t \leq
\sup_{0<s<\varphi^{-1}(t)}f^*(s)\varphi(s) = \varphinormqof{\chi_{(0,\varphi^{-1}(t))}f^*}.
\end{split}
\end{equation}
We continue by estimating $\varphinormqof{f_0}$. By definition of $f_0$,
\begin{equation}
\label{e2}
\begin{split}
\varphinormqof{f_0} = \varphinormqdefof{f_0}
\leq \sup_{0<s<\varphi^{-1}(t)}f^*(s)\varphi(s)
= \varphinormqof{\chi_{(0,\varphi^{-1}(t))}f^*}.
\end{split}
\end{equation}
Combining (\ref{e1}) and (\ref{e2}) we obtain
\begin{equation}
\label{leq}
K(f,t;\quasilorentz,L^\infty) \leq 2 \varphinormqof{\chi_{(0,\varphi^{-1}(t))}f^*},
\end{equation}
establishing the second inequality in~\eqref{E:K-inequality}. For the first one, once again, fix $f \in (\quasilorentz + L^\infty)$ non-negative and let $f=g+h$, where $g \in \quasilorentz$ and $h \in L^\infty$. Firstly, we shall assert that
\begin{equation}
\label{e3}
f^*(t) \leq g^*(t) + {\lVert h \rVert}_\infty \quad \text{for every $t\in(0,\infty)$.}
\end{equation}
For $t \in(0,\infty)$, set $\lambda = g^*(t) + {\lVert h \rVert}_\infty$ and
$y=\lvert\{s \in (0,\infty),f(s)>\lambda\}\rvert$. Then
\begin{equation*}
\begin{split}
y &= \lvert\{s \in (0,\infty),g(s) + h(s)>\lambda\}\rvert \\
&= \lvert\{s \in (0,\infty),g(s) + h(s)>g^*(t) + {\lVert h \rVert}_\infty\}\rvert \\
&\leq \lvert\{s \in (0,\infty),g(s)>g^*(t)\}\rvert +
\lvert\{s \in (0,\infty), h(s)>{\lVert h \rVert}_\infty\}\rvert \\
&= \lvert\{s \in (0,\infty),g(s)>g^*(t)\}\rvert,
\end{split}
\end{equation*}
since the set $\{s \in (0,\infty), h(s)>{\lVert h \rVert}_\infty\}$ obviously has zero measure. By \eqref{E:estimate-of-distribution-function} we obtain $y\leq t$. By definition of the decreasing rearrangement we get (\ref{e3}). Consequently, from subadditivity of supremum and because $\varphi$ is increasing, we obtain
\begin{equation}
\label{geq}
\begin{split}
&\sup_{0<s<\varphi^{-1}(t)}f^*(s)\varphi(s)\leq
\sup_{0<s<\varphi^{-1}(t)}g^*(s)\varphi(s) +
\sup_{0<s<\varphi^{-1}(t)}{\lVert h \rVert}_\infty\varphi(s) \\
&\leq \sup_{0<s< \infty}g^*(s)\varphi(s) +
{\lVert h \rVert}_\infty \varphi(\varphi^{-1}(t))
=\|g\|_{m_{\varphi}}+t\|h\|_{\infty}.
\end{split}
\end{equation}
Taking infimum over all such representations $f=g+h$, we arrive at
\begin{equation*}
\varphinormqof{\chi_{(0,\varphi^{-1}(t))}f^*}\leq K(f,t;\quasilorentz,L^\infty),
\end{equation*}
as desired. The assertion now follows from the combination of (\ref{leq}) and (\ref{geq}).
\end{proof}

\section{Operator \texorpdfstring{$S_a$}{} on rearrangement-invariant spaces}\label{CH:3}

\begin{definition}
For $a \in \measurablep$ we define the operator $S_a$ by the formula
\begin{equation*}
S_a f (t) = \int_0^\infty a(st)f(s) \diff s
\end{equation*}
for those $f \in \measurable$ for which the integral on the right is defined.
\end{definition}

Notice that $S_a$ is a~generalization of the Laplace transform, the~Laplace transform is $S_a$ for the choice $a(t) = e^{-t}$. In this section we will formulate a~characterisation of whether a~target space for the operator $S_a$ exists with a~domain space fixed. We will also show that whenever a~target space exists, optimal space also exists and we can find an~implicit definition of said space using the kernel $a$. We will then attempt to find equivalent definitions of target and optimal spaces using Calder\'on type operators rather then the~function $a$. We shall need some preliminary work. First, note that if
$X$ and $Y$ are rearrangement-invariant spaces and $a\in\measurable$, then
\begin{equation}\label{E:Self-Adjoint-main}
    S_a\colon X \rightarrow Y \iff S_a\colon Y' \rightarrow X'.
\end{equation}
This follows from the fact that $S_a$ is a self-adjoint operator, that is
\begin{equation}\label{E:self-adjoint}
    \int_0^\infty S_a f g = \int_0^\infty f S_a g \quad \text{for all $f,g\in \measurable$}.
\end{equation}

\begin{lemma}
\label{L:Tf-Estimate-With-Asterix}
Let $a$ be a non-increasing, non-negative function on $(0,\infty)$. Then
\begin{equation*}
(S_af)^* \leq S_a(f^*)\quad \text{for all $f \in \measurable$}.
\end{equation*}
\end{lemma}

\begin{proof}
Taking $t>0$ and $f \in \measurable$, we obtain, by the Hardy--Littlewood inequality (recall that $a$ is non-increasing), that
\begin{equation*}
\begin{split}
(S_af)^*(t) &= (|S_af|)^*(t) \leq (S_a|f|)^*(t) = \int_0^\infty |f(s)|a(st)\diff s \\
& \leq \int_0^\infty f^*(s)a(st)\diff s = S_a(f^*)(t).
\end{split}
\end{equation*}
\end{proof}

\begin{theorem}
\label{T:Range-partner-condition}
Let $X$ be a rearrangement-invariant space and let $a \in \measurable$ be non-increasing. Then $S_a\colon X \to L^1+L^\infty$ if and only if $a^{**}\in X'$. If that be the case, then the expression $\lVert S_a(f^*)\rVert_{X'}$ defines a rearrangement-invariant norm. Let $Z$ be the rearrangement-invariant space such that $\lVert f \rVert_{Z'}=\lVert S_a (f^*)\rVert_{X'}$. We have $S_a\colon X \to Z$.
\end{theorem}

\begin{proof}
We have $a^{**}=S_a(\chi_{(0,1)})$ and by definition
\begin{equation*}
    \lVert a^{**} \rVert_{X'}=\sup_{\lVert f \rVert_X\leq 1}\int_0^\infty |fa^{**}|.
\end{equation*}
Since $a^{**}$ is non-increasing, it follows from the fact that $S_a$ is self-adjoint that
\begin{equation*}
    \begin{split}
         \lVert a^{**} \rVert_{X'}&=\sup_{\lVert f \rVert_X\leq 1}\int_0^\infty f^*a^{**}
         = \sup_{\lVert f \rVert_X\leq 1}\int_0^\infty f^*S_a(\chi_{(0,1)})\\
         &=\sup_{\lVert f \rVert_X\leq 1}\int_0^\infty S_a(f^*)\chi_{(0,1)}
         =\sup_{\lVert f \rVert_X\leq 1}\lVert S_a(f^*)\rVert_{L^1+L^\infty},
    \end{split}
\end{equation*}
where the last equality follows from \eqref{E:K-traditional} and the fact that $S_a(f^*)$ is non-increasing. Now we get from Lemma \ref{L:Tf-Estimate-With-Asterix} that
\begin{equation*}
    \lVert a^{**} \rVert_{X'}=\sup_{\lVert f \rVert_X\leq 1}\lVert S_a(f)\rVert_{L^1+L^\infty}=\lVert S_a \rVert_{X\to L^1+L^\infty}.
\end{equation*}
The first equivalence follows. The proof of the fact that the expression $\lVert S_a (f^*) \rVert_{X'}$ defines a rearrangement-invariant norm is standard and we omit it.

 Since $S_a f^*$ is a non-increasing function we have by Lemma \ref{L:Tf-Estimate-With-Asterix}, the Hardy-Littlewood inequality, \eqref{E:self-adjoint} and the H\"older inequality that
\begin{equation*}
\begin{split}
    \lVert S_a f \rVert_Z
    &= \lVert (S_a f)^* \rVert_Z
    \leq \lVert S_a f^* \rVert_Z
    = \sup_{\lVert g\rVert_{Z'}\leq 1} \int_0^\infty S_a f^* g^* \\
    &= \sup_{\lVert g\rVert_{Z'}\leq 1} \int_0^\infty f^* S_a g^*
    \leq \sup_{\lVert g\rVert_{Z'}\leq 1} \lVert f \rVert_X \lVert S_a g^* \rVert_{X'}
    =\lVert f \rVert_X.
\end{split}
\end{equation*}
In other words, $S_a\colon X \rightarrow Z$.

\end{proof}
Norms defined using the kernel of some integral operator, similarly as in the preceding theorem, were studied before. For a differently defined operator of a type similar to $S_a$ this approach was used for example in \cite[Theorem~2.2]{Kerman} (only in the proof), but optimality of the resulting space (see Corollary \ref{C:Optimal-Space-Using-a}) is not shown. This is the subject of the following remark

\begin{remark}
\label{R:Range-partner-using-R}
Let $a\in\measurablep$ be non-increasing and non-trivial. Let $X$ be a rearrangement-invariant space with $a^{**}\in X'$ and $R$ be an operator defined at least on $\measurablep$. Define
\begin{equation*}
    \rho(f) = \lVert R f^* \rVert_{X'},
\end{equation*}
and assume that this $\rho$ is a rearrangement-invariant norm. Define $Z$ as the rearrangement-invariant space determined by $\rho'$. Consider the following two properties of $R$.
\begin{enumerate}
    \item [(i)] There exists $C>0$ such that for all $f\in\measurable$ it holds that
    \begin{equation*}
        S_a f^*(t)C \leq R f^*(t) \quad \text{for a.e. $t \in (0,\infty)$}
    \end{equation*}
    \item [(ii)] There exists $C>0$ such that for all $f\in\measurable$ it holds that
    \begin{equation*}
        S_a f^*(t)C \geq R f^*(t) \quad \text{for a.e. $t \in (0,\infty)$}
    \end{equation*}
\end{enumerate}
If (i) holds, then $S_a\colon X \rightarrow Z$. If (ii) holds and $Y$ is a rearrangement-invariant space such that $S_a\colon X \rightarrow Y$, then $Z\hookrightarrow Y$.
\end{remark}

The proof of the remark is easy and left to the reader.
Setting $R = S_a$ and using Theorem \ref{T:Range-partner-condition} and Remark \ref{R:Range-partner-using-R}, we arrive at the following corollary.
\begin{corollary}\label{C:Optimal-Space-Using-a}
Let $a\in\measurablep$ be non-increasing and let $X$ be a rearrangement-invariant space such that $a^{**}\in X'$. The rearrangement-invariant $Z$ space satisfying  $\lVert f \rVert_{Z'} = \lVert S_a f^*\rVert_{X'}$, $f\in\measurable$, is the smallest rearrangement-invariant space for which $S_a\colon X \rightarrow Z$.
\end{corollary}
\begin{definition}
Assume that $X$ is a~rearrangement-invariant space. We denote by $E(X)$ the function given by
\[
E(X)(t) = \frac{t}{\rVert E_t \rVert_{X\to X}}\quad\text{for $t\in(0,\infty)$.}
\]
\end{definition}
\begin{theorem}
\label{T:boundedness-of-Sa-general}
If $X$ is a rearrangement-invariant space, then $m_{E(X)}$ is well defined. Let $a \in \measurablep$ be non-increasing. If $a\in X'$ then
\begin{equation*}
S_a\colon X\to m_{E(X)}.
\end{equation*}
\end{theorem}

\begin{proof}
The proof of the fact that $E(X)$ is quasi-concave is standard. 
It follows therefore that $E(X)$ satisfies 
the~$\Delta_2$~condition. 
Indeed since 
$t\mapsto \tfrac{E(X)(t)}{t}$ is non-increasing, one has
\begin{equation*}
    E(X)(2t) = \frac{E(X)(2t)}{2t}2t\leq \frac{E(X)(t)}{t}2t,
\end{equation*}
which is the~$\Delta_2$~condition. We have shown that $m_{E(X)}$ is well defined. Now assume that $a \in X'$ is non-increasing.
Let $f \in \measurablep$. We first note that since $a$ is non-increasing, so is
$S_a f$. Thus, using the change of variables $st = u$ and the H\"older inequality, we obtain
\begin{equation*}
\begin{split}
\lVert S_a f \rVert_{m_{E(X)}} =
&\smallmnormdefof{\EofXint{X}}{S_a f(t)} =
\sup_{t\in(0,\infty)} \frac{t}{\rVert E_t \rVert_{X}} \aoperator \\
=&\sup_{t\in(0,\infty)} \frac{t}{\rVert E_t \rVert_{X}} \frac{1}{t} \int_0^\infty a(u)E_t f (u) \diff u \\
\leq & \sup_{t\in(0,\infty)} \frac{1}{\rVert E_t \rVert_{X}}
\lVert E_t f \rVert_X \lVert a \rVert_{X'}
\leq \lVert a \rVert_{X'}  \lVert f \rVert_X.
\end{split}
\end{equation*}
\end{proof}

%

\begin{definition}
\label{D:R-infty}
Let $X$ be a rearrangement-invariant space such that $\varphi_{X'}$, the fundamental function of $X'$, is strictly increasing and unbounded. Let $\varphi \colon (0,\infty) \to (0, \infty)$ be quasi-concave, unbounded function with $\varphi(0_+) = 0$. Then we define the function $\psi$ with the formula
\begin{equation}
    \label{E:def-psi}
    \psi(t) =\varphi_{X'}^{-1}\left(\frac{1}{\varphi(t)}\right) \quad \text{for $t\in(0,\infty)$.}
\end{equation}
For $f \in \measurable$ and $t\in(0,\infty)$ we define the functions $\alpha(f)$, $\beta(f)$ with the formulas
\begin{equation}
    \label{E:def-alpha}
    \alpha(f)(t) = \int_0^{\psi(t)} f^*(s) \diff s,
\end{equation}
\begin{equation}
    \label{E:def-beta}
    \beta(f)(t) = \frac{1}{\varphi(t)}\lVert f^*\chi_{(\psi(t), \infty)} \rVert_X,
\end{equation}
and we set $R_\infty^X(f) = \alpha(f) + \beta(f)$.
\end{definition}

\begin{definition}\label{D:R-1}
Let $X$ be a rearrangement-invariant space such that $\varphi_X$ is strictly increasing and unbounded. Let $\varphi$ be a function such that if we define the function $\overline{\varphi}$
with the formula $\overline{\varphi}(u) = \tfrac{u}{\varphi(u)}$, then the function $\overline{\varphi}$ is quasi-concave and unbounded with $\overline{\varphi}(0+) = 0$. Then we define the function $\tilde{\psi}$ with the formula
\begin{equation*}
    \tilde{\psi}(t) = \varphi_X^{-1}\left(\frac{1}{\overline{\varphi}(t)}\right).
\end{equation*}
For $f \in \measurable$ and $t \in (0,\infty)$ we define the function $R_1^Xf(t)$ with the formula
\begin{equation*}
    R_1^X(f)(t) = \frac{1}{\varphi(t)} \lVert f^*\chi_{(0,\tilde{\psi}(t))} \rVert_X.
\end{equation*}
\end{definition}
In the following theorems we will show that, under some conditions, both of the operators $R_1^X$, $R_\infty^X$ define rearrangement-invariant spaces in a way we would expect from Remark \ref{R:Range-partner-using-R}. Note that both $R_1^X$ and $R_\infty^X$ are defined only by means of the space $X$ and the function $\varphi$. If we define $\varphi = E(X)$ then $R_1^X$ and $R_\infty^X$ are defined only through $X$. This means that the obtained range partners depend only on $X$ and the fixed domain space, but not directly on the kernel $a$. These spaces, of course, are candidates for the optimal range partner, but in the general setting of rearrangement-invariant spaces, it seems very difficult to prove that they, under reasonable conditions, indeed are optimal. We can, however, obtain optimality of said spaces in the special case of $a \in L^\infty\cap L^1$, see Theorem \ref{T:Optimal-space-a-integrable}, or when the domain space is a Lorentz space, with its exponents lying within some boundaries. This is further explored in Section \ref{CH:4}.

\begin{lemma}
\label{T:L-gives-norm}
Assume that $X$ and $\varphi$ are as in Definition \ref{D:R-infty}. Let $W$ be a  rearrangement-invariant space and set
\begin{equation*}
    \rho(f) = \lVert R_\infty^X f \rVert_W, \, f\in\measurablep.
\end{equation*}
If the condition
\begin{equation}
\label{ChiInSpaceCondition}
\min\left\{\frac{1}{\varphi},1\right\} \in W
\end{equation}
holds, then $\rho$ is a rearrangement-invariant Banach function norm.
\end{lemma}

\begin{proof}
To prove the triangle inequality, fix $f_1, f_2 \in \measurablep$ and $t \in (0,\infty)$.
From the definition of the associate norm we have
\begin{equation*}
    R_\infty^X (f_1+f_2)(t) =
    \int_0^{\psi(t)} (f_1 + f_2)^* +
    \frac{1}{\varphi(t)} \sup\left\{\int_{\psi(t)}^\infty (f_1+f_2)^*g^*, \,
    \lVert g \rVert_{X'}\leq 1\right\}.
\end{equation*}
Take arbitrary $g \in X'$ with $\lVert g \rVert_{X'} \leq 1$ and set
\begin{equation*}
    h(s) = \begin{cases} 1,                             &s\in(0,\psi(t)) \\
                         \frac{1}{\varphi(t)}g^*(s)     &s\in[\psi(t),\infty)
           \end{cases}.
\end{equation*}
We know that $X' \hra M(X')$ with the norm of the embedding equal to 1. Therefore we have
\begin{equation*}
    \sup_{s\in(0,\infty)}g^*(s)\varphi_{X'}(s) \leq
    \sup_{s\in(0,\infty)}g^{**}(s)\varphi_{X'}(s) \leq
    \lVert g \rVert_{X'} \leq
    1.
\end{equation*}
In particular for $s = \psi(t)$ we have
\begin{equation*}
    g^*(\psi(t)) \leq \frac{1}{\varphi_{X'}(\psi(t))} =
    \frac{1}{\varphi_{X'}\varphi_{X'}^{-1}(\frac{1}{\varphi(t)})} =
    \varphi(t),
\end{equation*}
thus $h$ is non-increasing. Now thanks to the subadditivity of $f \mapsto f^{**}$ we have
\begin{equation*}
    \int_0^u (f_1 + f_2)^* \leq \int_0^u f_1^* + f_2^*, \quad u \in (0,\infty).
\end{equation*}
Using Hardy's Lemma \eqref{Hardy-Lemma} we obtain
\begin{equation} \label{E:Hardy-Consequence}
    \int_0^\infty (f_1 + f_2)^*h \leq \int_0^\infty f_1^*h + f_2^*h.
\end{equation}
From the definition of $h$ it is clear that
\begin{equation*}
    \int_0^{\psi(t)} (f_1 + f_2)^* + \frac{1}{\varphi(t)} \int_{\psi(t)}^\infty (f_1+f_2)^*g^* =
    \int_0^\infty (f_1 + f_2)^*h,
\end{equation*}
which, in combination with \eqref{E:Hardy-Consequence}, gives
\begin{equation*}
    \begin{split}
    &\int_0^{\psi(t)} (f_1 + f_2)^* + \frac{1}{\varphi(t)} \int_{\psi(t)}^\infty (f_1+f_2)^*g^*\leq \\
    &\int_0^{\psi(t)} f_1^* + \frac{1}{\varphi(t)} \int_{\psi(t)}^\infty f_1^*g^* +
    \int_0^{\psi(t)} f_2^* + \frac{1}{\varphi(t)} \int_{\psi(t)}^\infty f_2^*g^*.
    \end{split}
\end{equation*}
Since this holds for all $g \in X'$ with $\lVert g \rVert_{X'} \leq 1$, we have
\[
R_\infty^X (f_1+f_2)(t) \leq R_\infty^X (f_1)(t) + R_\infty^X (f_2)(t) \quad \text{for $t \in (0,\infty)$},
\]
which, using the (P2) property of the norm in $W$, gives the triangle inequality.

The fact that $\rho(f) = 0 \iff f=0$ holds trivially. Positive homogeneity is trivial. We've shown that (P1) holds. Next, (P6) holds obviously and (P2) and (P3) are direct consequences of the corresponding properties of $f \mapsto f^*$ and of $\lVert \cdot \rVert_X$ and $\lVert \cdot \rVert_{W}$.

To show (P4) we only need to show that $\rho(\chi_{(0,u)}) < \infty$ for any $u\in (0,\infty)$ because $\rho$ is defined in terms of the non-increasing rearrangement and we know, that for a measurable set $E \subset (0,\infty)$ it holds that $(\chi_E)^* = \chi_{(0,|E|)}$. Fix $u \in (0,\infty)$,
by the definition of $\alpha$ and $\beta$ we have for $t \in (0,\infty)$
\begin{equation*}
\alpha(\chi_{(0,u)})(t) = \min\{\psi(t), u\} \leq u\min\{\psi(t),1\}
\end{equation*}
and
\begin{equation*}
\begin{split}
\beta(\chi_{(0,u)})(t) = \frac{1}{\varphi(t)}
\lVert \chi_{(0,u) \cap (\psi(t),\infty)} \rVert_X =
\frac{1}{\varphi(t)} \varphi_X(\max\{u-\psi(t), 0\})
\leq \frac{1}{\varphi(t)} \varphi_X(u),
\end{split}
\end{equation*}
where $\varphi_X$ denotes the fundamental function of X. Furthermore, since $\lim_{t \rightarrow 0_+} \psi(t) = \infty$, there exists $\epsilon > 0$ such that
$\beta(\chi_{(0,u)})(t) = 0$, for $t<\epsilon$, and $\frac{1}{\varphi(\epsilon)}\geq 1$. Since $\frac{1}{\varphi} \leq
\frac{1}{\varphi(\epsilon)}$ on $(\epsilon,\infty)$ we have
\begin{equation*}
\beta(\chi_{(0,u)}) \leq \varphi_X(u) \min\left\{\frac{1}{\varphi},\frac{1}{\varphi(\epsilon)}\right\} \leq
\varphi_X(u) \frac{1}{\varphi(\epsilon)} \min\left\{\frac{1}{\varphi},1\right\}.
\end{equation*}
Thus, according to (\ref{ChiInSpaceCondition}), $\beta(\chi_{(0,u)}) \in W$. To show that $\alpha(\chi_{(0,u)}) \in W$ we need to only show that there exists a constant $C$ such that $\min\{\psi,1\} \leq C\min\{\frac{1}{\varphi},1\}$, for which it is sufficient to show that there exists $t>0$ and $C>0$ such that for all $s>t$
\[
\frac{1}{\varphi(s)} C\geq \psi(s) = \varphi_{X'}^{-1}(\frac{1}{\varphi(s)}),
\]
which is equivalent to
$\varphi_{X'}(\tau C) \geq \tau$, for $ \tau  \leq \frac{1}{\varphi(t)}$, which follows from quasi-concavity of $\varphi_{X'}$. Indeed, each quasi-concave function dominates the function $\min\{1,\tau\}$, $\tau \in (0,\infty)$, up to a multiplicative constant, and so we have
\[
\varphi_{X'}(C\tau) \geq C_1 C \tau,
\]
for all $C$ and $\tau$ such that $C\tau \leq 1$ and for some $C_1$. If we set $C = \frac{1}{C_1}$, then we have
\[
\varphi_{X'}(C\tau) \geq \tau \quad \text{for $\tau \leq \frac{1}{C}$}.
\]
We only need to find $t$ such that $\frac{1}{\varphi(t)} \leq \frac{1}{C}$, which we can do since $\varphi$ is unbounded. Now we have $t$ and $C$ as we wanted, therefore we have just proven that $\min\{1,\psi\}$ is dominated by $\min\{1,\frac{1}{\varphi}\}$ up to a constant.
Now since both $\beta(\chi_{(0,u)}) \in W$ and $\alpha(\chi_{(0,u)}) \in W$, it obviously holds that $\rho(\chi_{(0,u)})<\infty$ and thus (P4) holds.

It remains to show (P5). Let $E \subset (0,\infty)$ be of finite measure, let $f \in \measurablep$ and choose $u\in\psi^{-1}(|E|)$. Then, since $\psi$ is non-increasing, one has
\begin{equation*}
\begin{split}
\rho(f) \geq
&\lVert \alpha(f) \rVert_W \\ \geq
&\lVert \chi_{(0,u)}(t) \int_0^{\psi(t)} f^*(s) \diff s \rVert_{W} \\ \geq
&\lVert \chi_{(0,u)} \rVert_{W} \int_0^{|E|} f^*(s) \diff s \\ \geq
&\lVert \chi_{(0,u)} \rVert_{W} \int_E f(s) \diff s,
\end{split}
\end{equation*}
which establishes (P5).
\end{proof}
\begin{theorem}
\label{MainTheorem}
Assume that $X$ and $\varphi$ are as in Definition \ref{D:R-infty}. Let $Y \subset X + L^1$ be a rearrangement-invariant space. 
Let $a$ be a~non-increasing non-negative function on $(0,\infty)$ such that
\begin{equation}
\label{square}
\begin{split}
&S_a \colon X \rightarrow m_\varphi, \\
&S_a \colon L^1 \rightarrow L^\infty.
\end{split}
\end{equation}
Assume that \eqref{ChiInSpaceCondition} holds for $W=Y'$ and let $Z$ be the rearrangement-invariant space satisfying 
\begin{equation*}
    \lVert f \rVert_{Z'}=\lVert R_\infty^X f \rVert_{Y'}, \, f\in \measurable.
\end{equation*}
Then
\begin{equation*}
S_a\colon Y \rightarrow Z.
\end{equation*}
\end{theorem}

\begin{proof}
The fact that such $Z$ is necessarily a rearrangement-invariant space was proved in the preceding lemma, so we only need to show $S_a\colon Y \rightarrow Z$.
To this end, we will need to calculate the $K$-functionals of spaces $(X,L^1)$ and spaces $(m_\varphi,L^\infty)$. By Theorem~\ref{P-Linf-Smallm-K-functional} we have
\begin{equation*}
K(f,t;\quasilorentz,L^\infty)  \approx \varphinormqof{\chi_{(0,\varphi^{-1}(t))}f^*},
\end{equation*}
and, by \cite[Theorem 5.1]{Milman}, we have
\begin{equation*}
K(f,t,L^1,X) \leq C \int_0^{\varphi_{X{'}}^{-1}(t)} f^*(s) \diff s +
t \lVert f^* \chi_{(\varphi_{X{'}}^{-1}(t), \infty)} \rVert_X,
\end{equation*}
for some constant $C>0$.
Fix arbitrary $f \in L^1 + L^\infty$ and $t \in (0,\infty)$. By (\ref{square}) and the definition of the $K$-functional, we have
\begin{equation*}
K(S_a f,t,L^\infty,\quasilorentz) \leq C K(f,t,L^1,X)
\end{equation*}
for some constant $C>0$.
Combining that with the well known equality
\begin{equation*}
\frac{1}{t} K(f,t,L^\infty,\quasilorentz) = K(f,\frac{1}{t}, \quasilorentz, L^\infty),
\end{equation*}
we obtain
\begin{equation*}
\begin{split}
&\sup_{0<u<\varphi^{-1}(\frac{1}{t})} (Tf)^*(u)\varphi(u) \leq \\
&\leq \frac{C}{t} \left( \int_0^{\varphi_{X{'}}^{-1}(t)} f^*(s) \diff s +
t \lVert f^* \chi_{(\varphi_{X{'}}^{-1}(t), \infty)} \rVert_X \right).
\end{split}
\end{equation*}
In particular, since $\varphi$ is quasiconcave and therefore continuous and $(S_a f)^*$ is non-increasing, we can take $u = \varphi^{-1}(\frac{1}{t})$ and obtain
\begin{equation*}
\begin{split}
&\frac{1}{t}(S_a f)^*(\varphi^{-1}(\frac{1}{t})) \leq \\
&\leq \frac{C}{t} \left( \int_0^{\varphi_{X{'}}^{-1}(t)} f^*(s) \diff s +
t \lVert f^* \chi_{(\varphi_{X{'}}^{-1}(t), \infty)} \rVert_X \right).
\end{split}
\end{equation*}
Now since $\varphi$ is a~one-to-one mapping on $(0,\infty)$ and the above holds for every
$t \in (0,\infty)$, substituting $\varphi^{-1}(\frac{1}{t}) = u$ we obtain, for all
$u \in (0,\infty)$,
\begin{equation*}
\varphi(u)(S_a f)^*(u)
\leq C \varphi(u) \int_0^{\psi(u)} f^*(s) \diff s +
C\lVert f^* \chi_{(\psi(u),\infty)} \rVert_X.
\end{equation*}
Dividing by $\varphi(u)$ and changing the variable $u$ to $t$ yields the following result. There exists $C>0$ such that for all $f\in X+L^1$ and $t>0$
\begin{equation*}
(S_a f)^*(t) \leq CR_\infty^X f(t),
\end{equation*}
which is the property (i) in Remark \ref{R:Range-partner-using-R}, therefore $S_a\colon Y\rightarrow Z$.
\end{proof}

In conjunction with Theorem \ref{T:boundedness-of-Sa-general} we can now formulate a corollary of the preceding theorem, which makes the result more manageable.

\begin{corollary}\label{MainCorollary}
Let $X$ be such that $\varphi = E(X)$ is a quasi-concave unbounded strictly increasing function with $\varphi(0_+) = 0$. Let $Y \subset X + L^1$ be a~rearrangement-invariant space and let $a\in X' \cap L^\infty$ be non-increasing. Assume that \eqref{ChiInSpaceCondition} holds for $W=Y'$ and let $Z$ be the rearrangement-invariant space satisfying 
\begin{equation*}
    \lVert f \rVert_{Z'}=\lVert R_\infty^X f \rVert_{Y'}, \, f\in \measurable.
\end{equation*}
Then
$S_a\colon Y \to Z$.
\end{corollary}

\begin{proof}
Since $a \in L^\infty = (L^1)'$, $E(L^1)\equiv 1$ and $m_1 = L^\infty$, we have $S_a\colon L^1 \to L^\infty$, and since $a \in X'$ and $E(X) = \varphi$, we have $S_a\colon X \to m_\varphi$ thanks to Theorem \ref{T:boundedness-of-Sa-general}. Now we can apply Theorem \ref{MainTheorem}.
\end{proof}

\begin{lemma}
Assume that $X$ and $\varphi$ are as in Definition \ref{D:R-1}. Let $W$ be a rearrangement-invariant space and set
\[
\rho(f) = \lVert R_1^X(f) \rVert_W, \, f\in \measurable.
\]
If the condition
\begin{equation}
\label{Norm-Condition-L1}
    \min\left\{\frac{1}{\varphi(t)},\frac{1}{t}\right\} \in W
\end{equation}
holds, then $\rho$ is a rearrangement-invariant Banach function norm.
\end{lemma}

\begin{proof}
To show that (P1) holds, we need only to show the triangle inequality, since the other two assertions clearly hold. To that end take $f_1, f_2 \in \measurable$ and $t \in (0,\infty)$. We have
\begin{equation*}
    \begin{split}
        R_1^X(f_1+f_2)(t) &= \tfrac{1}{\varphi(t)} \lVert (f_1+f_2)^*\chi_{(0,\tilde{\psi}(t))} \rVert_X \\
        &= \tfrac{1}{\varphi(t)} \sup_{\lVert g \rVert_{X'} \leq 1} \int_0^{\tilde{\psi}(t)} (f_1+f_2)^*g^* \\
        &\leq \tfrac{1}{\varphi(t)} \sup_{\lVert g \rVert_{X'} \leq 1} \int_0^{\tilde{\psi}(t)} f_1^*g^* + \int_0^{\tilde{\psi}(t)} f_2^*g^* \\
        &= R_1^X(f_1)(t) + R_1^X(f_2)(t),
    \end{split}
\end{equation*}
where the only inequality is a~consequence of the Hardy Lemma \eqref{Hardy-Lemma}. Now the (P2) property of the norm in $W$ gives the triangle inequality. Properties (P2) and (P3) obviously hold for $\rho$. To show (P4) take $u \in (0,\infty)$. We have for all $t>0$
\begin{equation*}
\begin{split}
    \tfrac{1}{\varphi(t)}\lVert \chi_{(0,u)}\chi_{(0,\tilde{\psi}(t))} \rVert_X &=
    \tfrac{1}{\varphi(t)}\varphi_X(\min\{u,\tilde{\psi}(t)\}) \\
    &=\tfrac{1}{\varphi(t)} \min\{\varphi_X(u),\varphi_X(\tilde{\psi}(t)) \\
    &=\tfrac{1}{\varphi(t)} \min\{\varphi_X(u),\tfrac{1}{\overline{\varphi}(t)}\},
\end{split}
\end{equation*}
where we have used the fact that $\varphi_X$ is non-decreasing and the definition of $\tilde{\psi}$. If we move $\tfrac{1}{\varphi(t)}$ inside the argument of the minimum, we get from the definition of $\overline{\varphi}$ that
\begin{equation*}
    \tfrac{1}{\varphi(t)} \min\{\varphi_X(u),\tfrac{1}{\overline{\varphi}(t)}\} =
    \min\{\tfrac{\varphi_X(u)}{\varphi(t)},\tfrac{1}{t}\} \leq
    C \min\{\tfrac{1}{\varphi(t)},\tfrac{1}{t}\}
\end{equation*}
for some $C>0$ depending on $u$ but not on $t$. Now \eqref{Norm-Condition-L1} gives (P4) property of $\rho$. It remains to show (P5). To that end take $E\subset (0,\infty)$ with positive measure, $f\in\measurablep$ and choose $u\in\tilde{\psi}^{-1}(|E|)$. Since $\tilde{\psi}$ is non-increasing, we have
\begin{equation*}
    \begin{split}
        \rho(f)
        &= \lVert \tfrac{1}{\varphi(t)}\lVert f^*\chi_{(0,\tilde{\psi}(t)})\rVert_X\rVert_W \\
        &\geq \lVert \chi_{(0,u)}(t) \tfrac{1}{\varphi(t)}\lVert f^*\chi_{(0,\tilde{\psi}(t)})\rVert_X\rVert_W \\
        &\geq \lVert \chi_{(0,u)} \rVert_W \tfrac{1}{\varphi(u)}\lVert f^*\chi_{(0,|E|)} \rVert_X \\
        &\geq C \lVert \chi_{(0,u)} \rVert_W \tfrac{1}{\varphi(u)} \int_0^{|E|}f^*\chi_{(0,|E|)} \\
        &\geq C \lVert \chi_{(0,u)} \rVert_W \tfrac{1}{\varphi(u)} \int_E f
    \end{split}
\end{equation*}
for some $C>0$ only depending on $E$ and the space $X$, where the second to last inequality is the (P5) property of the norm in $X$. Since $u$ does not depend on $f$, we have obtained the (P5) property of $\rho$.
\end{proof}

\begin{theorem}
Assume that $X$ and $\varphi$ are as in Definition \ref{D:R-1}. Let $Y \subset X  + L^\infty$ be a rearrangement-invariant space, assume that $\overline{\varphi}$ is strictly increasing.
Let $a$ be a non-increasing non-negative function on $(0,\infty)$ such that
\begin{equation}
\begin{split}
\label{square2}
        &S_a\colon X \rightarrow m_\varphi, \\
        &S_a\colon L^\infty \rightarrow L^{1,\infty}.
\end{split}
\end{equation}
Assume \eqref{Norm-Condition-L1} holds for $W = Y'$ and let $Z$ be the rearrangement-invariant space for which
\[
\lVert f \rVert_{Z'} = \lVert R_1^X f \rVert_{Y'}, \, f \in \measurable.
\]
Then
\[
S_a\colon Y \rightarrow Z.
\]
\end{theorem}

\begin{proof}
Easy modifications of Theorem \ref{P-Main-Kfunc} and \cite[Theorem~4.1]{Milman} give respectively
\begin{equation} \label{E:KIneq1}
\sup_{u < (\overline{\varphi})^{-1}(t)} (S_a f)^*(2u)u + t\sup_{u \geq (\overline{\varphi})^{-1}(t)} (S_a f)^*(2u)\varphi(u) \leq
    C_1 K(S_a f, t, L^{1,\infty}, m_\varphi)
\end{equation}
and
\begin{equation} \label{E:KIneq2}
    K(f, t, L^\infty, X) \leq C_2(f^*(\varphi_X^{-1}(\tfrac{1}{t}))
    + t\lVert f^*\chi_{(0,\varphi_X^{-1}(\tfrac{1}{t}))} \rVert_X),
\end{equation}
for all $t \in (0,\infty)$, $f \in \measurable$ and some constants $C_1,C_2>0$.
Positive homogeneity and the definition of the fundamental function immediately gives
\[
\lVert f^*\chi_{(0,\varphi_X^{-1}(\tfrac{1}{t}))} \rVert_X \geq
f^*(\varphi_X^{-1}(\tfrac{1}{t})) \varphi_X(\varphi_X^{-1}(\tfrac{1}{t})),
\]
therefore we obtain
\begin{equation*}
     K(f, t, L^\infty, X) \leq (C_2+1)(t\lVert f^*\chi_{(0,\varphi_X^{-1}(\tfrac{1}{t}))} \rVert_X).
\end{equation*}
From the definition of the $K$-functional and \eqref{square2}, we obtain
\begin{equation*}
    K(S_a f, t, L^{1,\infty}, m_\varphi) \leq C K(f, t, L^\infty, X),
\end{equation*}
for some constant $C>0$ and all $f \in \measurable$ and $t \in (0,\infty)$.
Using \eqref{E:KIneq1} and \eqref{E:KIneq2} we have a possibly different constant $C>0$ such that, for all $f \in \measurable$ and $t\in(0,\infty)$,
\begin{equation*}
    \sup_{u < (\overline{\varphi})^{-1}(t)} (S_a f)^*(2u)u + t\sup_{u \geq (\overline{\varphi})^{-1}(t)} (S_a f)^*(2u)\varphi(u)
    \leq C t\lVert f^*\chi_{(0,\varphi_X^{-1}(\tfrac{1}{t}))} \rVert_X.
\end{equation*}
In particular, we can clearly replace the second supremum with the value of its argument in $(\overline{\varphi})^{-1}(t)$. Furthermore, we can do the same in the first supremum since $S_a(f^*)$ is non-increasing and $u\mapsto 2u$ is continuous. This results in the following inequality
\begin{equation*}
    (S_a f)^*(2(\overline{\varphi})^{-1}(t))(\overline{\varphi})^{-1}(t) +
    t(S_a f)^*(2(\overline{\varphi})^{-1}(t))\varphi((\overline{\varphi})^{-1}(t)) \leq
    C t\lVert f^*\chi_{(0,\varphi_X^{-1}(\tfrac{1}{t}))} \rVert_X
\end{equation*}
for all $f \in \measurable$ and $t \in (0,\infty)$.
Now we can substitute $t = \overline{\varphi}(s)$, since $\overline{\varphi}$ is a one to one mapping of $(0,\infty)$, which results in
\begin{equation*}
    (S_a f)^*(2s)s + \overline{\varphi}(s)(S_a f)^*(2s)\varphi(s)
    \leq C\overline{\varphi}(s) \lVert f^*\chi_{(0,\tilde{\psi}(s))} \rVert_X
\end{equation*}
for all $f \in \measurable$ and $s \in (0,\infty)$.
A simple calculation now gives
\begin{equation*}
    (S_a f)^*(2s) \leq C\tfrac{1}{\varphi(s)}\lVert f^*\chi_{(0,\tilde{\psi}(s))} \rVert_X
\end{equation*}
for all $f \in \measurable$ and $s \in (0,\infty)$,
which implies the following statement. There exists $C>0$ such that for all $f \in X+L^\infty$ and $t>0$
\begin{equation*}
(S_a f)^*(t) \leq CR_1^X f(\tfrac{1}{2}t).
\end{equation*}
Now we need only to use Remark \ref{R:Range-partner-using-R} the fact that the dilation operator is bounded on every rearrangement-invariant space.

\end{proof}
The following corollary is analogous to Corollary \ref{MainCorollary}.
\begin{corollary}\label{MainCorollary2}
Let $X$ be such that if $\varphi = E(X)$, then $\overline{\varphi}$ is a quasi-concave unbounded strictly increasing function with $\overline{\varphi}(0_+) = 0$. Let $Y \subset X + L^\infty$ be a rearrangement-invariant space and let $a \in X' \cap L^1$ be non-increasing and non-negative. Assume \eqref{Norm-Condition-L1} holds for $W = Y'$ and let $Z$ be the rearrangement-invariant space for which
\[
\lVert f \rVert_{Z'} = \lVert R_1^X f \rVert_{Y'}, \, f \in \measurable.
\]
Then $S_a\colon Y \rightarrow Z$.

\end{corollary}
Now we explore the properties of $S_a$ when the kernel $a$ is both integrable and bounded. This is indeed the strongest possible condition, at least in the context of rearrangement-invariant spaces, as $L^1\cap L^\infty$ is the smallest rearrangement-invariant space. The chosen approach uses the operator $R_\infty^X$ and the space $X=L^\infty$. Working with $R_1^X$ and setting $X=L^1$, unincidentally, yields the same result.
\begin{lemma}
\label{L:Reversed-Inequality}
Let $a: (0,\infty) \rightarrow (0,\infty)$ be non-increasing and non-zero on at least some set of non-zero measure. Then there is a constant $C$ such that, for all $f \in \measurable$,
\begin{equation*}
S_a (f^*) (t) \geq C \int_0^{\frac{1}{t}} f^*(s) \diff s\quad\text{for all $t\in(0,\infty)$.}
\end{equation*}
\end{lemma}

\begin{proof}
Since $a$ is non-increasing and not zero on at least some set of non-zero measure, there exists $u\in(0,1)$ such that $\inf_{s<u} a(s) = C_1 >0$.
Now if $f \in \measurable$ and $t\in(0,\infty)$, then
\begin{equation*}
\begin{split}
S_a (f^*)(t) & = \int_0^\infty a(st)f^*(s) \diff s
\geq \int_0^{\frac{u}{t}} a(st)f^*(s) \diff s \\
& \geq C_1 \int_0^{\frac{u}{t}} f^*(s) \diff s
\geq C_1u \int_0^{\frac{1}{t}} f^*(s) \diff s,
\end{split}
\end{equation*}
where the last inequality follows from the fact that $f^{**}$ is non-increasing.
\end{proof}

\begin{lemma}
\label{L:Equiv-L1intLinfty}
Let $X=L^\infty$ and $\varphi = E(L^\infty)$, that is $\varphi(t) = t$, $t \in (0,\infty)$. Then we have, for all $f \in L^1 + L^\infty$ and $t>0$,
\begin{equation*}
\int_0^{\frac{1}{t}} f^* \leq
R_\infty^X f(t) \leq
2 \int_0^{\frac{1}{t}} f^*.
\end{equation*}
\end{lemma}
\begin{proof}
Take $f \in L^1 + L^\infty$ and $t>0$.
First, we shall observe that $\varphi_{X'}(t) = t$, since
$X' = L^1$, therefore $\psi(t) = \frac{1}{t}$. This means that
\[
\alpha(f)(t) = \int_0^{\frac{1}{t}} f^*
\]
and, since $\beta(f) \geq 0$, we can easily obtain the first inequality from the definition of $R$. Now, a~simple calculation shows that
\[
\beta(f)(t) = \frac{1}{t} \lVert f^*\chi_{(\frac{1}{t},\infty)} \rVert_\infty =
\frac{1}{t}f^*\left(\frac{1}{t}\right) \leq
\frac{1}{t}f^{**}\left(\frac{1}{t}\right) =
\alpha(f)(t).
\]
Therefore
\[
R_\infty^X(t) = \alpha(f)(t) + \beta(f)(t) \leq \alpha(f)(t) + \alpha(f)(t) =
2 \int_0^{\frac{1}{t}} f^*,
\]
which establishes the second inequality.
\end{proof}
What follows is a result of independent interest in the spirit of \cite[Theorem~3.4]{BEP}.
\begin{theorem}
\label{T:Optimal-space-a-integrable}
Let $a \in L^\infty \cap L^1$ be non-trivial, non-negative and non-increasing. Let $Y$ be a rearrangement-invariant space such that
\begin{equation}
\label{LinftyL1cond}
\min\left\{1, \frac{1}{t} \right\} \in Y'.
\end{equation}
For $f \in \measurable$ and $t>0$, set
\[
\alpha(f)(t) = \int_0^{\frac{1}{t}} f^*
\]
and
\[
\rho(f) = \lVert \alpha(f) \rVert_{Y'}.
\]
Then $\rho$ is a rearrangement-invariant norm such that if we set $Z$ to be the \\ rearrangement-invariant space given by $\rho'$, then
\[
S_a\colon Y \rightarrow Z,
\]
and $Z$ is optimal for $S_a$ and $Y$. Furthermore, if the condition \eqref{LinftyL1cond} does not hold, then there is no rearrangement-invariant space $Z$ such that $S_a: Y \rightarrow Z$.
\end{theorem}
\begin{proof}
The proof of the fact that $\rho$ is a rearrangement-invariant norm is almost identical to that of Theorem~\ref{MainTheorem}, see also \cite[Proposition 3.3]{BEP}, and therefore omitted.
We thus know from Corollary \ref{MainCorollary} and Lemma \ref{L:Equiv-L1intLinfty} that $S_a: Y \rightarrow Z$.

The optimality of $Z$ is a direct consequence of Lemma \ref{L:Reversed-Inequality} and Remark \ref{R:Range-partner-using-R}. Indeed, from Lemma \ref{L:Reversed-Inequality} we have a constant $C>0$ such that, for all $f \in \measurable$ and $t \in (0,\infty)$,
\[
S_a (f^*) (t) \geq C \int_0^{\frac{1}{t}} f^*(s) \diff s,
\]
whence Remark \ref{R:Range-partner-using-R} gives the optimality of $Z$.

It remains to show that if (\ref{LinftyL1cond}) does not hold, then there is no rearrangement-invariant space $Z$ such that $S_a: Y \rightarrow Z$.
Since $a$ is integrable, bounded, non-zero on some set of non-zero measure and non-increasing, we can find a constant $C'>0$ such that
\[
a^{**}(t) \geq
C' \min\left\{1, \frac{1}{t} \right\}
\quad \text{for all $t>0$}.
\]
Indeed, suppose first that $t\le 1$. Then we find $u>0$ such that $\inf_{s \in (0,u)} a(s) = C_1 > 0$. If
$u \geq 1$, then simply
\[
\frac{1}{t} \int_0^t a(s) \diff s \geq \frac{1}{t} t C_1 = C_1.
\]
If $u<1$, then we have
\[
\frac{1}{t} \int_0^t a(s) \diff s \geq \frac{1}{t} \int_0^{\min\{t,u\}} a(s) \diff s
\geq \min\{\frac{1}{t} t C_1, \frac{1}{t} u C_1\} = u C_1.
\]
Now, let $t>1$. Then we set $C_2 = \int_0^1 a(s)\diff s$ and observe that
\[
\frac{1}{t} \int_0^t a(s) \diff s \geq \frac{1}{t} C_2.
\]
Combining these estimates, we arrive at
\[
a^{**}(t) \geq C' \min\left\{1, \frac{1}{t}\right\} \quad \text{for all $t>0$}.
\]
where $C'=\min\{C_1,uC_1,C_2\}$. Therefore $a^{**}\not\in Y'$, whence Theorem \ref{T:Range-partner-condition} implies that there is no rearrangement-invariant space $Z$ such that $S_a\colon Y \to Z$. Indeed this holds since $L^1+L^\infty$ is the largest among all rearrangement-invariant spaces.
\end{proof}

\section{Operator \texorpdfstring{$S_a$}{} on the scale of Lorentz spaces} \label{CH:4}

In this section we shall formulate some sufficient conditions involving the kernel function $a$ under which we can obtain the optimal range partner $Z$ for $S_a$ and a fixed domain Lorentz space $Y$. It turns out that the space $Z$ is also a Lorentz space and, perhaps aside from extremal cases, its exponents only depend on the exponents of the space $Y$ and not on the kernel $a$. 

\begin{theorem}
Let $p\in(1,\infty]$ and $\xi\in(1,\infty)$. Define the operators $S_p$ and $T_p$ by
\begin{equation*}
    (S_p f)(t)=t^{-\tfrac{1}{p}}\int_0^{\tfrac{1}{t}}f(s)s^{-\tfrac{1}{p}}\diff s
\end{equation*}
and
\begin{equation*}
    (T_p f)(t)=\int_0^{\tfrac{1}{t}} f(s)\diff s+ t^{-\tfrac{1}{p}}\int_{\tfrac{1}{t}}^\infty f(s) s^{-\tfrac{1}{p}}\diff s,     
\end{equation*}
for $t\in(0,\infty)$ and all $f\in\measurable$ for which the right sides make sense.
\begin{itemize}
    \item[(i)] It holds that
    \begin{equation}\label{E-R-1-estimate-edge-case}
        S_p \colon L^{p',1} \to L^{p,\infty};\quad S_p \colon L^{\infty} \to L^{1,\infty}.
    \end{equation}
    \item [(ii)] If $\xi<p$ and $\eta \in [1,\infty]$, then 
    \begin{equation}\label{E-R-1-estimate}
        S_p\colon L^{\xi',\eta} \to L^{\xi,\eta}.
    \end{equation}
    \item [(iii)] 
    It holds that
        \begin{equation}\label{E-R-infty-estimate-edge-case}
            T_p\colon L^{p',1}\to L^{p,\infty};\quad T_p\colon L^{1}\to L^\infty.
        \end{equation}
    \item[(iv)]
    If $p<\xi$ and $\eta\in [1,\infty]$, then 
    \begin{equation}\label{E-R-infty-estimate}
        T_p \colon L^{\xi',\eta} \to L^{\xi,\eta}.
    \end{equation}
    \item[(v)] The following inequality holds for all $f\in\measurable$, $\xi \in (1,\infty)$ and $\eta \in [1,\infty]$
    \begin{equation}\label{E:Rho-in-Lorentz-Lower-Estimate}
    \lVert f \rVert_{\xi',\eta} \leq \left\lVert \int_0^{\tfrac{1}{t}} f^*\right\rVert_{\xi,\eta}.
    \end{equation}
\end{itemize}

\end{theorem}
\begin{proof}
    We shall prove statements \textit{(i)}-\textit{(iv)} as the statement \textit{(v)} is very simple and can be found in a more complete form in \cite[Proposition 3.7]{BEP}. Let $f\in\measurable$. It then holds that $|S_p(f)|\leq S_p |f|$ whenever the function on the right is finite a.e. Hence we may assume $f\in\measurablep$. Since $S_p(f)$ is non-increasing it follows that
    \begin{equation*}
        \lVert S_p f \rVert_{L^{p,\infty}}=\sup_{t\in(0,\infty)}\int_0^{\tfrac{1}{t}}f(s)s^{-\tfrac{1}{p}}\diff s = \lVert f \rVert_{L^{p',1}}.
    \end{equation*}
    On the other hand, let $f\in L^{\infty}$. Then it holds that
    \begin{equation*}
    \begin{split}
        \lVert S_p f\rVert_{L^{1,\infty}}&=\sup_{t\in(0,\infty)}t^{1-\tfrac{1}{p}}\int_{0}^{\tfrac{1}{t}}f(s)s^{-\tfrac{1}{p}} \diff s
        \leq \lVert f \rVert_{L^{\infty}}\sup_{t\in(0,\infty)} t^{1-\tfrac{1}{p}}(\tfrac{1}{t})^{1-\tfrac{1}{p}}\tfrac{1}{1-\tfrac{1}{p}}\\
        &=p'\lVert f \rVert_{L^\infty}.
    \end{split}    
    \end{equation*}
    We have shown that \textit{(i)} holds. It follows immediately from the Marcinkiewicz interpolation theorem (see e.g. \cite[Theorem 4.4.13]{BS}) that \textit{(ii)} holds.
    
   We also have $|T_p(f)|\leq T_p(|f|)$ whenever the function on the right is finite a.e. Hence we may once again assume $f\in\measurablep$ to show \textit{(iii)}. Notice that $T_p(f)$ is non-increasing as it may be written in the form
   \begin{equation*}
       T_p(f)(t)=\int_0^\infty \min\left\{1,(ts)^{-\tfrac{1}{p}}\right\} f^*(s) \diff s.
   \end{equation*}
   From the definition and by sub-additivity of supremum we have
   \begin{equation*}
   \begin{split}
       \lVert T_p(f)\rVert_{L^{p,\infty}}&\leq\sup_{t\in(0,\infty)} t^{\tfrac{1}{p}}\int_0^{\tfrac{1}{t}}f(s)\diff s + \sup_{t\in(0,\infty)} t^{\tfrac{1}{p}}t^{-\tfrac{1}{p}}\int_{\tfrac{1}{t}}^\infty f(s) s^{-\tfrac{1}{p}}\diff s\\
       &=\lVert S_{\infty} f \rVert_{L^{p,\infty}}+\lVert f \rVert_{L^{p',1}}.
   \end{split}
   \end{equation*}
   Now by \textit{(ii)} in case $p<\infty$ or by \textit{(i)} in case $p=\infty$ we have a constant $C>0$ such that $\lVert S_\infty(f) \rVert_{L^{p,\infty}}\leq C \lVert f \rVert_{L^{p',1}}$ and \textit{(iii)} follows. Moreover, due to Marcinkiewicz interpolation theorem \textit{(iv)} holds. This completes the proof.
\end{proof}

Now that we have established some necessary estimates, it is time to use them to calculate the optimal range partners for $S_a$, in the case when the domain space is a Lorentz space. Let us first fix some notation. If it is stated what space $X$ is, we automatically assume that $\varphi=E(X)$ and that the operators $R_\infty^X$ and $R_1^X$ are defined as in Definitions \ref{D:R-infty} and \ref{D:R-1}. Furthermore by $Z_\infty$ and $Z_1$ we shall denote the spaces given by norms associated to $\lVert R_\infty^X f \rVert_{Y'}$ and $\lVert R_1^X f \rVert_{Y'}$ respectively, where $Y$ is the fixed domain space.
\begin{lemma}
\label{L:Optimal-Space-Lower-Constraint}
Let $\xi\in(1,\infty)$, $\eta\in[1,\infty]$ and $a \in \measurablep$ a non-increasing function. If $W$ is a rearrangement-invariant space such that $S_a\colon L^{\xi,\eta} \rightarrow W$, then $L^{\xi',\eta}\hookrightarrow W$.
\end{lemma}
\begin{proof}
    Let
    \begin{equation*}
        Tf(t) = \int_0^{\tfrac{1}{t}} f^*(s) \diff s \quad \text{for $f \in \measurable$, $t>0$},
    \end{equation*}
    then by Theorem \ref{T:Optimal-space-a-integrable} we know that the expression $\lVert Tf \rVert_{\xi',\eta'}$ defines a rearrangement-invariant norm since $\min\{1,\tfrac{1}{t}\}\in L^{\xi',\eta'}$. Furthermore from Lemma \ref{L:Reversed-Inequality} we have a constant $C>0$ such that $S_a f^* \geq CTf$ for all $f\in\measurable$. If we now set $Z$ to be the rearrangement-invariant space determined by the norm associate to $\lVert Tf \rVert_{\xi',\eta'}$, then Remark \ref{R:Range-partner-using-R} gives $Z\hookrightarrow W$. Furthermore \eqref{E:Rho-in-Lorentz-Lower-Estimate} gives $Z' \hookrightarrow L^{\xi,\eta'}$ or equivalently $L^{\xi',\eta} \hookrightarrow Z$, whence $L^{\xi',\eta} \hookrightarrow W$.
\end{proof}
\begin{theorem}\label{T:Optimal-Partner-L1-Lorentz}
    Let $p \in [1,\infty)$ and $a \in \measurablep$ a non-increasing function such that $a \in L^1\cap L^{p',\infty}$. If $\xi \in (1,\infty)$, $p <\xi $, $\eta \in [1,\infty]$, then $S_a\colon L^{\xi,\eta} \rightarrow L^{\xi',\eta}$ and furthermore, $L^{\xi',\eta}$ is the optimal range partner for $S_a$ and $L^{\xi,\eta}$.
\end{theorem}
\begin{proof}
    First set $X=L^{p,1}$ and $Y=L^{\xi,\eta}$. Then since $p<\infty$ we have up to a constant that $\overline{\varphi}(t) = t^{1/p}$, $t>0$ which is a quasi-concave, unbounded, strictly increasing function with $\overline{\varphi}(0_+) = 0$. Furthermore $Y=L^{\xi,\eta} \subset X+L^\infty$ since $p<\xi$, also $a \in X'\cap L^1$ and the condition \eqref{Norm-Condition-L1} holds for $W=L^{\xi',\eta'}$. Therefore Corollary \ref{MainCorollary2} gives $T\colon~L^{\xi,\eta}~\rightarrow~Z_1$.
    It is easy to calculate that $\tilde{\psi}(t) = \tfrac{1}{t}$ and so by \eqref{E-R-1-estimate} we have a constant $C>0$ such that for all $f\in\measurable$
    \begin{equation*}
        \lVert f \rVert_{Z_1'} = \left\lVert t^{-\tfrac{1}{p}}\int_0^{\tfrac{1}{t}} f^*(s)s^{-\tfrac{1}{p}}\diff s\right\rVert_{\xi',\eta'} \leq C  \lVert f \rVert_{\xi,\eta'}.
    \end{equation*}
    In other words $L^{\xi,\eta'}\hookrightarrow Z_1'$ or equivalently $Z_1 \hookrightarrow L^{\xi',\eta}$, therefore $S_a\colon L^{\xi,\eta} \rightarrow L^{\xi',\eta}$.
    The optimality of $L^{\xi',\eta}$ is now an immediate consequence of Lemma \ref{L:Optimal-Space-Lower-Constraint}.
\end{proof}

\begin{theorem}\label{T:Optimal-Partner-Linfty-Lorentz}
    Let $q \in(1,\infty]$ and $a\in\measurablep$ a non-increasing function such that $a\in L^1 \cap L^{q',\infty}$. If $\xi\in(1,\infty)$, $\xi<q$, $\eta \in [1,\infty]$ then $S_a\colon L^{\xi,\eta} \rightarrow L^{\xi',\eta}$ and furthermore, $L^{\xi',\eta}$ is the optimal range partner for $S_a$ and $L^{\xi,\eta}$.
\end{theorem}
\begin{proof}
    First set $X=L^{q,1}$ and $Y=L^{\xi,\eta}$. Then since $q>1$ we have up to a constant that $\varphi(t)=t^{\tfrac{1}{q'}}$, $t>0$, which is a quasi-concave, unbounded, strictly increasing function. Furthermore $Y=L^{\xi,\eta} \subset X+L^1$ since $\xi<q$, also $a \in L^\infty \cap X'$ and the condition \eqref{RINormCondition} holds for $W=Y'=L^{\xi',\eta'}$. Therefore Corollary \ref{MainCorollary2} gives $S_a \colon L^{\xi,\eta'} \rightarrow Z_\infty$. It is easy to calculate that $\psi(t)=\tfrac{1}{t}$, $t>0$, and by \eqref{E-R-infty-estimate} we have a constant $C>0$ such that for all $f\in\measurable$
    \begin{equation*}
        \lVert f \rVert_{Z_\infty'} =
        \left\lVert \int_0^{\tfrac{1}{t}}f^*(s)\diff s + t^{-\tfrac{1}{q'}}\int_{\tfrac{1}{t}}^\infty f^*(s) s^{-\tfrac{1}{q'}} \diff s \right\rVert_{\xi',\eta'} \leq C \lVert f \rVert_{\xi,\eta'}.
    \end{equation*}
    In other words $L^{\xi,\eta'}\hookrightarrow Z_\infty'$ or equivalently $Z_\infty \hookrightarrow L^{\xi',\eta}$, therefore $S_a\colon L^{\xi,\eta} \rightarrow L^{\xi',\eta}$. The optimality of $L^{\xi',\eta}$ is now an immediate consequence of Lemma \ref{L:Optimal-Space-Lower-Constraint}.
\end{proof}

Now that we have these optimal results for $a\in L^\infty$ and $a\in L^1$ it is time to apply them to a general $a\in \measurablep$. We recall that by Theorem \ref{T:Range-partner-condition} we require $a$ to be in at least some rearrangement-invariant space, otherwise there is no hope of finding a range partner for $S_a$ and any fixed rearrangement-invariant domain space. Since $L^1 + L^\infty$ is the largest rearrangement-invariant space we can assume $a\in L^1+L^\infty$, which allows us to divide $a$ into its integrable and bounded parts and then use the previous theorems and the fact, that $S_a$ is linear in $a$.

\begin{definition}
Given $a\in L^1+L^\infty$ we define the functions $a_1$, $a_\infty$ with following formulas
\begin{equation*}
    \begin{split}
        a_1(t) &= (a(t)- a(1))\chi_{(0,1)}\quad \text{for $t>0$} \\
        a_\infty(t)&=\min\{a(1),a(t)\}\quad \text{for $t>0$}.
    \end{split}
\end{equation*}
\end{definition}
Note that $a = a_1 + a_\infty$ and that by \eqref{E:K-traditional} we have $a_1 \in L^1$ and $a_\infty \in L^\infty$.

\begin{theorem}\label{T:Optimal-Partner-general-Lorentz}
    Let $a\in L^1+L^\infty$ be non-increasing. Let $p\in[1,\infty)$, $q\in(1,\infty]$ be such that $a_1\in L^{p',\infty}$, $a_\infty \in L^{q',\infty}$. If $\xi\in (p,q)$, $\eta \in [1,\infty]$, then $S_a\colon L^{\xi,\eta} \rightarrow L^{\xi',\eta}$ and furthermore, $L^{\xi',\eta}$ is the optimal range partner for $S_a$ and $L^{\xi,\eta}$.
\end{theorem}
\begin{proof}
   From Theorems \ref{T:Optimal-Partner-L1-Lorentz} and \ref{T:Optimal-Partner-Linfty-Lorentz} we have
    \begin{equation*}
        \begin{split}
            &S_{a_1}\colon L^{\xi,\eta} \rightarrow L^{\xi',\eta} \\
            &S_{a_\infty}\colon L^{\xi,\eta} \rightarrow L^{\xi',\eta},
        \end{split}
    \end{equation*}
    respectively, with both of these range spaces being optimal. Now since $a_1+a_\infty=a$, we also have $S_{a_1}+S_{a_\infty}=S_a$, which implies $S_a\colon L^{\xi,\eta} \rightarrow L^{\xi',\eta}$ with $L^{\xi',\eta}$ being the optimal range partner for $S_a$ and $L^{\xi,\eta}$.
\end{proof}

If we set $A=\{r\in [1,\infty),\, a_1 \in L^{r',\infty}\}$, $B=\{r\in (1,\infty],\, a_\infty \in L^{r',\infty}\}$, then it is easy to see that $A$, $B$ are some intervals. Theorem \ref{T:Optimal-Partner-general-Lorentz} gives the optimality of $S_a\colon L^{\xi,\eta} \rightarrow L^{\xi',\eta}$ for all $\xi \in (A\cap B)^o$. It is also easy to see that if $\xi \not\in A\cap B$, then there is no range partner for $S_a$ and $L^{\xi',\eta}$ simply because there is no hope of $a^{**}\in L^{\xi,\eta}$. The problematic case is when $A\cap B = [p,q]$ for some $p,q\in(1,\infty)$ and $\xi=p$ or $\xi=q$. It is not hard, however, to find an example that shows that a similar result in this case does not hold. If one considers $a(t) = t^{-\tfrac{1}{p'}}\chi_{(0,1)}$, $t>0$, then we have $A = [p,\infty)$, $B=(1,\infty]$, thus $A \cap B = [p,\infty)$. Setting $\xi=p$, $\eta=1$ we have $a^{**}\in L^{\xi',\eta'}=L^{p',\infty}$ and so by Corollary \ref{C:Optimal-Space-Using-a} we know that the optimal space is given by norm associate to
\begin{equation*}
    \lVert S_a f^* \rVert_{p',\infty} = \sup_{t>0} t^{\tfrac{1}{p'}}t^{-\tfrac{1}{p'}}
    \int_0^{\tfrac{1}{t}} f^*(s)s^{\tfrac{1}{p'}} = \lVert f \rVert_{p,1}.
\end{equation*}
This means that for this specific $a$ we have the optimal result $S_a\colon L^{\xi,\eta} \to L^{\xi',\eta'}$. All we can say, using \eqref{E-R-1-estimate-edge-case}, \eqref{E-R-infty-estimate-edge-case} and \eqref{E:Rho-in-Lorentz-Lower-Estimate} and the same approach as in the preceding theorems, is the following. If $A\cap B = [p,q]$, then the optimal range partner $X_{\text{opt}}^p$ for $S_a$ and $L^{p,1}$ satisfies $L^{p',1}\hra X_{\text{opt}}^p \hra L^{p',\infty}$. In that case, the optimal range partner $X_{\text{opt}}^q$ for $S_a$ and $L^{q,1}$ satisfies $L^{q',1}\hra X_{\text{opt}}^q \hra L^{q',\infty}$.

The preceding assertions sum up the boundedness and optimality properties of $S_a$ on the scale of Lorentz spaces $L^{\xi,\eta}$. It is a simple observation that as long as $a$ is non-trivial, $a^{**} \not\in L^1$ and therefore, by Theroem \ref{T:Range-partner-condition}, there is no range partner for $S_a$ and $L^\infty$. Combined with the preceding results and denoting by $p = \inf A$, $q=\sup B$, we know the following.
\begin{itemize}
    \item [(i)] If $\xi\in(p,q)$, $\eta\in[1,\infty]$, then a range partner for $S_a$ and the domain space $L^{\xi,\eta}$ exists. If that is the case, then the optimal range partner is the space $L^{\xi',\eta}$.
    \item[(ii)] If $\xi\in(1,\infty)$, $\xi \not\in [p,q]$, $\eta\in[1,\infty]$, then no range partner for $S_a$ and the domain space $L^{\xi,\eta}$ exists.
    \item[(iii)] If $p\not\in A$ or $q\not\in B$ and $\xi=p$ or $\xi=q$ respecitvely, then no range partner for $S_a$ and the domain space $L^{\xi,\eta}$ exists.
    \item[(iv)] If $p=\infty$ (or equivalently $A=\emptyset$), then no range partner for $S_a$ and the domain space $L^{\xi,\eta}$ exists.
    \item[(v)] $S_a\colon L^1 \to L^\infty$ if and only if $a\in L^\infty$ (or equivalently $a_1\in L^\infty$ or equivalently $p=1$), if that is the case, then $L^\infty$ is the optimal range partner for $S_a$ and $L^1$
    \item[(vi)] If $q=-\infty$ (or equivalently $B=\emptyset$), then there is no rearrangement-invariant domain Lorentz space $Y\not = L^\infty$, for which a range partner exists.
    \item[(vii)] If $\xi=\eta=\infty$ then no range partner for $S_a$ and the domain space $L^{\xi,\eta}=L^\infty$ exists.
\end{itemize}
The only points perhaps not explained above are (v) and (vi). Theorem \ref{T:boundedness-of-Sa-general} gives $S_a\colon L^1\to L^\infty$ if $a\in L^\infty$, whereas the optimality of this result is a consequence of Lemma \ref{L:Optimal-Space-Lower-Constraint} and the fact that
\[
\left\lVert \int_0^{\tfrac{1}{t}} f^* \right\rVert_\infty = \lVert f \rVert_1.
\]
If $a\not\in L^\infty$ then also $a^{**}\not\in L^\infty$, thus by Theorem \ref{C:Optimal-Space-Using-a} there is no range partner for $S_a$ and $L^1$. We have shown that (v) holds. To show (vi) it suffices to show that $a^{**} \not\in Y$, this is, however, a direct consequence of $q=-\infty$.

\paragraph{Acknowledgements}
I would like to thank the referees for careful reading of the paper and many valuable suggestions. I would also like to thank Lubo\v{s} Pick for stimulating discussions.

\end{document}